\documentclass[11pt,twoside,reqno]{amsart}
\allowdisplaybreaks
\usepackage{amsmath,amstext,amssymb,epsfig}
\usepackage{graphicx}
\usepackage{mathrsfs}
\calclayout
\makeatletter
\renewcommand{\@seccntformat}[1]{\bf\csname the#1\endcsname.}
\renewcommand{\section}{\@startsection{section}{1}
	\z@{.7\linespacing\@plus\linespacing}{.5\linespacing}
	{\normalfont\upshape\bfseries\centering}}
\renewcommand{\@biblabel}[1]{\@ifnotempty{#1}{#1.}}
\makeatother
\theoremstyle{plain}
\newtheorem{thm}{Theorem}[section]
\newtheorem{lem}[thm]{Lemma}
\newtheorem{prop}[thm]{Proposition}

\theoremstyle{definition}

\newtheorem{defn}[thm]{Definition}
\newtheorem{rem}{Remark}[section]

\usepackage{cancel}
\usepackage[parfill]{parskip}
\usepackage[german]{varioref}
\usepackage[all]{xy}
\usepackage{color}

\def \>{\succ}
\def \<{\prec}

\begin{document}	
	\title[Nil MANSURO\u{G}LU \textsuperscript{}, Bouzid Mosbahi \textsuperscript{}]{ Generalized derivations of BiHom-supertrialgebras}
	\author{ Nil MANSURO\u{G}LU \textsuperscript{1},\quad Bouzid Mosbahi \textsuperscript{2}}
\address{\textsuperscript{1}Kırşehir Ahi Evran University, Faculty of Arts and Science,
Department of Mathematics, 40100 Kırşehir, Turkey.}
   	\address{\textsuperscript{2} Department of Mathematics, Faculty of Sciences, University of Sfax, Sfax, Tunisia}
		
\email{\textsuperscript{1}nil.mansuroglu@ahievran.edu.tr}
\email{\textsuperscript{2}mosbahi.bouzid.etud@fss.usf.tn}
	
	
	
	\begin{abstract}


 In this note, our goal is to describe the concept of generalized derivations in the context of BiHom-supertrialgebras. We provide a comprehensive analysis of the properties and applications of these
generalized derivations, including their relationship with other algebraic structures. We also explore various examples and applications of BiHom-supertrialgebras in different fields
of mathematics and physics. Our findings contribute to a deeper understanding of the algebraic properties and applications of BiHom-supertrialgebras, and pave the way for further research in this area.
\end{abstract}
\maketitle \section{ Introduction}

BiHom-supertrialgebras are a generalization of supertrialgebras \cite{9}, which are algebraic structures that combine the properties of superalgebras and trialgebras. Supertrialgebras have found
applications in different fields of mathematics and physics, including quantum groups, representation theory  and supersymmetry.
The goal of this paper is to provide a comprehensive overview of generalized derivations and applications of BiHom-supertrialgebras. Generalized derivations are linear maps that satisfy
certain algebraic properties and they also have an  significant position in the study of algebraic structures. On the other hand,  BiHom-supertrialgebras  are supertrialgebras equipped with two compatible
derivations. Generalized derivations in the context of BiHom-supertrialgebras have many applications in various fields. These
applications include  the study of Lie superalgebras, the construction of new algebraic structures and the investigation of differential equations. For more details see \cite{1,2}. To provide a comprehensive understanding of the topic, we will begin by introducing the basic definitions and properties of supertrialgebras and BiHom-supertrialgebras by using \cite{4,9,11,15,16}. We will then
delve into the theory of generalized derivations and explore their properties in the context of BiHom-supertrialgebras.
Furthermore, we will discuss various applications of generalized derivations in BiHom-supertrialgebras. These applications include the construction of Lie superalgebras from BiHom-supertrialgebras,
the study of differential equations associated with BiHom-supertrialgebras, and the investigation of the cohomology theory of BiHom-supertrialgebras. For more details see \cite{10,13,14,17}.
Throughout the paper, we will provide detailed proofs and examples to illustrate the concepts and results. We will also highlight the connections between the different areas of study
and discuss the potential future directions of research in this field.

\section{ Preliminaries}

\begin{defn}
Given a superspace $S$. If three even bilinear maps $\dashv,\vdash,\bot: S \times S \longrightarrow S$ and an even superspace homomorphism $\gamma : S \longrightarrow S$ satisfy the following axioms
$$\begin{array}{ll}
\gamma(d \dashv q) = \gamma(d) \dashv \gamma(q),\\
\gamma(d \bot q) = \gamma(d) \bot \gamma(q),\\
(d \dashv q) \dashv \gamma(y) = \gamma(d) \dashv (q \vdash y),\\
(d \dashv q) \vdash \gamma(y) = \gamma(d) \vdash (q \vdash y),\\
(d \dashv q) \dashv \gamma(y) = \gamma(d) \dashv (q \bot y),\\
(d \dashv q) \bot \gamma(y) = \gamma(d) \bot (q \vdash y),\\
(d \bot q) \vdash \gamma(y) = \gamma(d) \vdash (q \vdash y),\\
\end{array}\quad\quad
\begin{array}{ll}
\gamma(d \vdash q) = \gamma(d) \vdash \gamma(q),\\
\quad (d \dashv q) \dashv \gamma(y) = \gamma(d) \dashv (q \dashv y),\\
\quad (d \vdash q) \dashv \gamma(y) = \gamma(d) \vdash (q \dashv y),\\
\quad (d \vdash q) \vdash \gamma(y) = \gamma(d) \vdash (q \vdash y),\\
\quad (d \bot q) \dashv \gamma(y) = \gamma(d) \bot (q \dashv y),\\
\quad (d \vdash q) \bot \gamma(y) = \gamma(d) \vdash (q \bot y),\\
\quad (d \bot q) \bot \gamma(y) = \gamma(d) \bot (q \bot y)
\end{array}$$
for every $d,q,y\in S$, then  a quintuple $(S,\dashv,\vdash,\bot,\gamma)$ is called a Hom-associative supertrialgebra.
\end{defn}
\begin{rem}
 By taking $\gamma = id$, we obtain the associative supertrialgebra. Any Hom-associative supertrialgebra becomes a Hom-associative superdialgebra without the operation $\bot$.
\end{rem}
\begin{defn}\label{def2}
Let $S$ be a superspace.  If three even bilinear mappings $\dashv,\vdash,\bot: S \times S \longrightarrow S$ and two even superspace homomorphisms $\gamma,\xi : S \longrightarrow S$ satisfy the next axioms
\begin{align*}
&\text{(i)}\quad \gamma\circ \xi = \xi \circ\gamma, \\
&\text{(ii)}\quad(d\dashv q)\dashv\xi(y)=\gamma(d)\dashv(q\vdash r)=\gamma(d)\dashv(q\bot y),\\
&\text{(iii)}\quad (d\dashv q)\dashv\xi(y)=\gamma(d)\vdash(q\dashv y),\\
&\text{(iv)}\quad (d\dashv q)\vdash\gamma(y)=\xi(d)\vdash(q\vdash y)=(d\bot q)\vdash\xi(y),\\
&\text{(v)}\quad (d\bot q)\dashv\xi(y)=\gamma(d)\bot(q\dashv y),\\
&\text{(vi)}\quad (d\dashv q)\bot\xi(y)=\gamma(d)\bot(q\vdash y),\\
&\text{(vii)}\quad (d\vdash q)\bot\xi(y)=\gamma(d)\vdash(q\bot y)
\end{align*}
for every $d,q,y \in S$, then   a 6-truple $(S,\dashv,\vdash,\bot,\gamma,\xi)$   is called a BiHom-associative supertrialgebra.
\end{defn}

\begin{rem}\label{rq1}
 $\gamma$ and $\xi$ are endomorphisms with  $\dashv, \vdash$ and $\bot$, namely,
\begin{align*}
&\gamma(d\dashv q)=\gamma(d)\dashv\gamma(q), \quad \xi(d\dashv q)=\xi(d)\dashv\xi(q),\nonumber\\
&\gamma(d\vdash q)=\gamma(d)\vdash\gamma(q), \quad \xi(d\vdash q)=\xi(d)\vdash\xi(q),\nonumber\\
&\gamma(d\bot q)=\gamma(d)\bot\gamma(q), \quad \xi(d\bot q)=\xi(d)\bot\xi(q)\nonumber
\end{align*}
for each  $d, q \in S$. Then,
we say $S$ is  a multiplicative BiHom-associative supertrialgebra.
\end{rem}
\begin{defn} \label{def1}
 A linear mapping
$$\pi : (S_1, \dashv, \vdash,\bot, \gamma, \xi)\rightarrow(S_2, \dashv',\vdash',\bot', \gamma', \xi')$$ with
$$\gamma'\circ \pi
=\pi\circ\gamma,\;\; \xi'\circ \pi=\pi\circ\xi$$
and
\begin{align*}
 \pi(d\dashv q)&=\pi(d)\dashv'\pi(q),\nonumber\\
\pi(d\vdash q)&=\pi(d)\vdash'\pi(q),\nonumber\\
\pi(d\bot q)&=\pi(d)\bot'\pi(q)\nonumber
\end{align*}
 for every  $d, q \in S_1$  is said to be a  morphism of BiHom-associative supertrialgebras.
\end{defn}



\begin{prop}\label{p1}
Let $(S, \dashv, \vdash,\bot ,\gamma, \xi)$ be a   BiHom-associative supertrialgebra of dimension $n$.     Given an invertible linear map $l: S\rightarrow S$.
Then, there exists an isomorphism with a  BiHom-associative supertrialgebra  $(S, \dashv', \vdash', \bot', l\gamma\xi^{-1}, l\xi l^{-1})$, where
$(\dashv', \vdash', \bot')=l\circ(\dashv, \vdash, \bot)\circ(l^{-1}\otimes l^{-1}$). Furthermore, if
the structure constants of $(\dashv, \vdash \bot)$ with the basis
$\left\{u_1,\dots,u_n\right\}$ are  $\left\{\phi^k_{ij}, \psi^k_{ij}, \delta^k_{ij}\right\}$,   then  the structure constants of $(\dashv', \vdash', \bot')$
 with  the basis $\left\{l(u_1),\dots,l(u_n)\right\}$ are same.
\end{prop}
\begin{proof}
Here, one axiom only will be proved.
For any invertible linear mapping $l: S \rightarrow S$,  $(S, \dashv', \vdash',  \bot',  l\gamma l^{-1}, l\xi l^{-1})$ is a
BiHom-associative supertrialgebra.
\begin{align*}
(u\dashv'v)\dashv'l\xi l^{-1}(w)
&=l\dashv(l^{-1}\otimes l^{-1})(u,v)\dashv'l\xi l^{-1}(w)\\
&=(l(l^{-1}\otimes l^{-1})(u\dashv v))\dashv'l\xi l^{-1}(w)\\
&=l(l^{-1}\otimes l^{-1})(l(l^{-1}\otimes l^{-1})(u\dashv v)\dashv l^{-1}\xi l(w)\\
&=l(l^{-1}(u)\dashv l^{-1}(v))\dashv l^{-1}\xi l(w)\\
&=l(\gamma l^{-1}(u)\dashv(l^{-1}(v)\vdash l^{-1}(w)))\\
&=l(l^{-1}\otimes l^{-1})(l\otimes l)\gamma l^{-1}(u)\dashv(l^{-1}\otimes l^{-1}(v\vdash w))\\
&=l(l^{-1}\otimes l^{-1})(l\gamma l^{-1}(u)\dashv(l(l^{-1}\otimes l^{-1})(v\vdash w)))\\
&=l\gamma l^{-1}(u)\dashv'(v\vdash'w)= l\gamma l^{-1}(u)\dashv'(v\bot'w).
\end{align*}
Thus, $(S, \dashv', \vdash', \bot',  l\gamma l^{-1}, l\xi l^{-1})$ becomes a BiHom-associative supertrialgebra.\\
Moreover,  for $\gamma$
\begin{align*}
l\gamma l^{-1}(u\dashv'v)
&=l\gamma l^{-1}l\dashv(l^{-1}\otimes l^{-1})(u,v)\\
&=l\gamma\dashv(l^{-1}\otimes l^{-1})(u,v)\\
&=l(\gamma l^{-1}(u)\dashv\gamma l^{-1}(v))\\
&=l((l^{-1}\otimes l^{-1})(l\otimes l)\gamma l^{-1}(u)\dashv\gamma l^{-1}(v))
=(l\gamma l^{-1}(u)\dashv'l\gamma l^{-1}(v))
\end{align*}
and for $\xi$
\begin{align*}
l\xi l^{-1}(u\dashv'v)
&=l\xi l^{-1} l\dashv(l^{-1}\otimes l^{-1})(u,v)\\
&=l\xi\dashv(l^{-1}\otimes l^{-1})(u,v)\\
&=l(\xi l^{-1}(u)\dashv\xi l^{-1}(v))\\
&=l((l^{-1}\otimes l^{-1})(l\otimes l)\xi l^{-1}(u)\dashv\xi l^{-1}(v))
=(l\xi l^{-1}(u)\dashv'l\xi l^{-1}(v)).
\end{align*}

Therefore, $l : (S, \dashv, \vdash, \bot,  \gamma, \xi)\rightarrow (S, \dashv', \vdash', \bot',  l\gamma l^{-1}, l\xi l^{-1})$ is a
BiHom-associative supertrialgebra morphism, as $$l\circ\dashv=l\circ\dashv\circ(l^{-1}\otimes l^{-1})\circ(l\otimes l)=l'\circ(l\otimes l)$$ and $$(l\gamma l^{-1})\circ l=l\circ\gamma, \quad (l\xi l^{-1})\circ l=l\circ\xi.$$ \\ It is obvious to see  that
$\left\{l(u_i), \dots, l(u_n)\right\}$ is a basis of $S$. For each $i,j=1,\dots,n$, we have

\begin{align*}
(l(u_i)\dashv l(u_j))
&=l(l^{-1}(u_i)\dashv l^{-1}(u_j))=l(u_i\dashv u_j)=\sum_{k=1}^n\delta^k_{ij}l(u_k).
\end{align*}
\end{proof}

\begin{prop}
Let $(S,  \dashv, \vdash, \bot, \gamma, \xi)$ be a BiHom-associative supertrialgebra and  $(S, \dashv', \vdash', \bot', l\gamma l^{-1}, l\xi l^{-1})$ be as described in Proposition \ref{p1}. If $\phi$ is an automorphism of $(S, \dashv, \vdash, \bot,  \gamma, \xi)$, then
$l\phi l^{-1}$ is an automorphism of $(S, \dashv',\vdash', \bot',  l\gamma l^{-1}, l\xi l^{-1})$.
\end{prop}

\begin{proof}
By taking $\alpha=l\gamma l^{-1}$, we have
$$l\phi l^{-1}\alpha=l\phi l^{-1} l\gamma l^{-1}=l\phi\gamma l^{-1}=l\gamma\phi l^{-1}=l\gamma l^{-1} l\phi l^{-1}=\gamma l\phi l^{-1}.$$
For $\xi$, we take $\sigma=l\xi l^{-1}$. Consequently, we get
$$l\phi l^{-1}\sigma=l\phi l^{-1} l\xi l^{-1}=l\phi\xi l^{-1}=l\xi\phi l^{-1}=l\xi l^{-1} l\phi l^{-1}=\sigma l\phi l^{-1}.$$
For every $p,v\in S,$
\begin{align*}
 l\phi l^{-1}(l(p)\dashv'l(v))&=l\phi l^{-1} l(p\dashv v)
 =l\phi(p\dashv v)=l(\phi(p)\dashv\phi(v))\\
 &=(l\phi(p)\dashv' l\phi(v))
 =(l\phi l^{-1}(l(p)\dashv'\phi l\phi^{-1}(\phi(v))).
 \end{align*}
 By Definition \ref{def1}, $l\phi l^{-1}$ becomes an automorphism of $(S, \dashv', \vdash',  \bot', \gamma l^{-1}, l\xi l^{-1})$.
 \end{proof}



\begin{prop}
  Given two
 BiHom-associative supertrialgebras $(\mathcal{N},  \dashv_{\mathcal{N}}, \vdash_{\mathcal{N}}, \bot_{\mathcal{N}}, \gamma_{\mathcal{N}}, \xi_{\mathcal{N}})$ and
$(\mathcal{M}, \dashv_{\mathcal{M}}, \vdash_{\mathcal{M}},\bot_{\mathcal{M}},  \gamma_{\mathcal{M}}, \xi_{\mathcal{M}})$. Then,  a BiHom-associative supertrialgebra structure becomes
on $\mathcal{N}\oplus \mathcal{M}$ with the bilinear mappings $\triangleleft, \triangleright, \ast : (\mathcal{N}\oplus \mathcal{M})^{\otimes 2}\rightarrow \mathcal{N}\oplus \mathcal{M}$
  defined by
$$(n_1+m_1)\dashv(n_2+m_2) :=n_1\dashv_\mathcal{N}n_2+m_2\dashv_\mathcal{M}m_2,$$
$$(n_1+m_1)\vdash(n_2+m_2) :=n_1\vdash_\mathcal{N}n_2+m_1\vdash_\mathcal{M}m_2,$$
$$(n_1+m_1)\bot(n_2+m_2) :=n_1\bot_\mathcal{N}n_2+m_1\bot_\mathcal{M}m_2$$
and the linear
 mappings $\alpha=\alpha_\mathcal{N}+\alpha_\mathcal{M},\, \beta=\beta_\mathcal{N}+\beta_\mathcal{M} : \mathcal{N}\oplus\mathcal{M}\rightarrow\mathcal{N}\oplus\mathcal{M}$ given by
$$(\alpha_\mathcal{N}+\alpha_\mathcal{M})(x+y) :=\alpha_\mathcal{N}(x)+\alpha_\mathcal{M}(y),\, (\beta_\mathcal{N}+\beta_\mathcal{M})(x+y) :=\beta_\mathcal{N}(x)+\beta_\mathcal{M}(y),\,
\forall(x,y)\in \mathcal{N}\times\mathcal{M}.
$$
Furthermore, if $\xi : \mathcal{N}\rightarrow\mathcal{M}$  is a linear mapping,
then,
$$ \xi : (\mathcal{N}, \dashv_\mathcal{N}, \vdash_\mathcal{N},\bot_\mathcal{N}, \alpha_\mathcal{N}, \beta_\mathcal{N})
\rightarrow(\mathcal{M}, \dashv_\mathcal{M}, \vdash_\mathcal{M},\bot_\mathcal{M},  \alpha_\mathcal{M}, \beta_\mathcal{M})$$
 is a morphism iff its graph  $\Gamma_\xi=\{(x, \xi(x)), x\in \mathcal{N}\}$  is a BiHom-subalgebra of $(\mathcal{N}\oplus\mathcal{M}, \triangleleft, \triangleright,\ast ,\alpha, \beta)$.
\end{prop}

\begin{proof}
  First we
 suppose that
$\xi : (\mathcal{N}, \dashv_\mathcal{N}, \vdash_\mathcal{N},\bot_\mathcal{N}, \alpha_\mathcal{N}, \beta_\mathcal{N})\rightarrow(\mathcal{M}, \dashv_\mathcal{M}, \vdash_\mathcal{M},\bot_\mathcal{M},\alpha_\mathcal{M}, \beta_\mathcal{M})$
is a morphism.
As a result, $$(n+\xi(n))\dashv(m+\xi(m))=(n\dashv_\mathcal{N}m+\xi(n)\dashv_\mathcal{M}\xi(m))=n\dashv_\mathcal{N}m+\xi(n\dashv_\mathcal{N}m),$$
     $$(n+\xi(n))\vdash(m+\xi(m))=(n\vdash_\mathcal{N}m+\xi(n)\vdash_\mathcal{M}\xi(m))=n\vdash_\mathcal{N}m+\xi(n\vdash_\mathcal{N}m),$$
		 $$(n+\xi(n))\bot(m+\xi(m))=(n\bot_\mathcal{N}m+\xi(n)\bot_\mathcal{M}\xi(m))=n\bot_\mathcal{N}m+\xi(n\bot_\mathcal{N}m).$$
 Moreover, as $\xi\circ\alpha_\mathcal{N}=\alpha_\mathcal{M}\circ\xi,$ and $\xi\circ\beta_\mathcal{N}=\beta_\mathcal{M}\circ\xi,$ we have
$$
(\alpha_\mathcal{N}\oplus\alpha_\mathcal{M})(n, \xi(n))=(\alpha_\mathcal{N}(n), \alpha_\mathcal{M}\circ\xi(n))=(\alpha_\mathcal{M}(n), \xi\circ\alpha_\mathcal{N}(n))
$$
and
$$
(\beta_\mathcal{N}\oplus\beta_\mathcal{M})(n, \xi(n))=(\beta_\mathcal{N}(n), \beta_\mathcal{M}\circ\xi(n))=(\beta_\mathcal{N}(n), \xi\circ\beta_\mathcal{N}(n)).
$$
This implies that  $\Gamma_\xi$ is closed $\alpha_\mathcal{N}\oplus\alpha_\mathcal{M}$ and $\beta_\mathcal{N}\oplus\beta_\mathcal{M}.$
 Therefore, $\Gamma_\xi$  is a subalgebra of $(\mathcal{N}\oplus\mathcal{M}, \dashv, \vdash,\bot, \alpha, \beta).$ \\
Now, let the graph $\Gamma_\xi\subset\mathcal{N}\oplus\mathcal{M}$ be  a subalgebra of
$(\mathcal{N}\oplus\mathcal{M}, \dashv, \vdash, \alpha, \beta)$, then we get
$$(n+\xi(n))\dashv(m+\xi(m))=(n\dashv_\mathcal{N}m+\xi(n)\dashv_\mathcal{M}\xi(m))\in \Gamma_\xi, $$
 $$(n+\xi(n))\vdash(m+\xi(m))=(n\vdash_\mathcal{N}m+\xi(n)\vdash_\mathcal{M}\xi(m))\in \Gamma_\xi,$$
$$(n+\xi(n))\bot(m+\xi(m))=(n\bot_\mathcal{N}m+\xi(n)\bot_\mathcal{M}\xi(m))\in \Gamma_\xi.$$
Furthermore, $(\alpha_\mathcal{N}\oplus\alpha_\mathcal{M})(\Gamma_\xi)\subset \Gamma_\xi$ and
 $(\beta_\mathcal{N}\oplus\beta_\mathcal{M})(\Gamma_\xi)\subset \Gamma_\xi$ show that
$$
(\alpha_\mathcal{N}\oplus\alpha_\mathcal{M})(n, \xi(n))=(\alpha_\mathcal{N}(n),\alpha_\mathcal{M}\circ\xi(n))\in \Gamma_\xi,\,(\beta_\mathcal{N}\oplus\beta_\mathcal{M})(n, \xi(n))
=(\beta_\mathcal{N}(n),\beta_\mathcal{M}\circ\xi(n))\in \Gamma_\xi.
$$
This is equivalent to the condition $\alpha_\mathcal{M}\circ\xi(n)=\xi\circ\alpha_\mathcal{N}(k),$ i.e., $\alpha_\mathcal{M}\circ\xi=\xi\circ\alpha_\mathcal{N}.$\\
Similarly, $\beta_\mathcal{M}\circ\xi=\xi\circ\beta_\mathcal{N}$. Hence, $\xi$ is a morphism of BiHom-associative trialgebras.
\end{proof}

\begin{defn}
Given  a BiHom-associative supertrialgebra $(S, \dashv, \vdash, \bot,  \gamma, \xi)$  with a linear  mapping $\lambda : S\longrightarrow  S$  with
$\gamma\circ \lambda=\lambda\circ\gamma$, $\xi\circ \lambda=\lambda\circ\xi$ and
\begin{align*}
\lambda(d)\vdash \lambda(v)&=\lambda(\lambda(d)\dashv v+d\dashv \lambda(v)+c(d\dashv v)),\\
\lambda(d)\dashv\lambda \lambda(v)&=\lambda(\lambda(d)\vdash v+d\vdash \lambda(v)+c(d\vdash v)),\\
\lambda(d)\bot \lambda(v)&=\lambda(\lambda(d)\bot v+d\bot \lambda(v)+c(d\bot v)),
 \end{align*}
with $c\in F$, $d,v\in S$, we say $S$ is  a Rota-Baxter BiHom-associative supertrialgebra and $\lambda$ is  a Rota-Baxter operator of weight $\rho$ on $S$.
\end{defn}


\begin{prop}
  Given a Rota-Baxter BiHom-associative supertrialgebra $(S, \ast, \gamma, \xi, \nu)$ and the new operations $\dashv, \vdash$  and $\bot$ on $S$ are described by
$d\dashv v=d\ast \lambda(v),\,d\vdash v= \lambda(d)\ast v$ and $d\bot v=c d\ast v$. Then, $(S,\dashv,\vdash, \bot, \gamma, \xi)$ is a BiHom-associative supertrialgebra.
\end{prop}
\begin{proof}
Here,   one axiom is only proved and the other axioms are showed in  similar way. Indeed, for any $p, v, r\in S,$
\begin{align*}
(d\dashv v)\vdash \xi (r)&=(d\ast \lambda(v))\vdash\xi(r)\\
&=\lambda(d\ast \lambda(v))\ast \xi(r)\\
&=\gamma \lambda(d)\ast(\lambda(v)\ast r)\\
&=\gamma \lambda(d)\ast(v\vdash r)\\
&=\gamma((d)\vdash(v\vdash r)=(d\bot v)\vdash \xi(r).
 \end{align*}
This ends the proof.
\end{proof}

\begin{prop}
Let $(S, \dashv,\vdash, \bot, \gamma, \xi)$ be a BiHom-associative supertrialgebra such that $\gamma^2=\xi^2\gamma\circ\xi=\xi\circ\gamma=id.$
Then,  $(S,\dashv,\vdash, \bot, \gamma, \xi)\cong (S, \dashv,\vdash, \bot, \xi,\gamma).$
\end{prop}
\begin{proof}
Here, the proof of only one axiom will be done. For any $d, v, r\in S,$
\begin{align*}
\gamma(d)\dashv (v\vdash r)&=\gamma(d)\dashv(v\bot r)\\
&\Leftrightarrow\gamma(\gamma\xi(d))\dashv (v\vdash r)=\gamma(\gamma\xi(d))\dashv(v\bot r)\\
&\Leftrightarrow\gamma^2\xi(d)\dashv (v\vdash r)=\gamma^2\xi(d)\dashv(v\bot r)\\
&\Leftrightarrow\xi(d)\dashv (v\vdash r)=\xi(d)\dashv(v\bot r).
 \end{align*}
Hence,  $(S, \dashv,\vdash,\bot, \gamma, \xi)\cong (S, \dashv,\vdash, \bot,  \xi,\gamma).$
\end{proof}

\begin{prop}
Let $(S, \dashv,\vdash, \bot, \gamma, \xi)$ be a BiHom-associative supertrialgebra. Then, $(S, \dashv, \bot, \ast, \gamma, \xi)$ is a BiHom-associative supertrialgebra
where  $d\ast v=d\vdash v+d\bot v$ for any $d, v, r\in S.$
\end{prop}
\begin{proof}
One axiom only will be proved. For any $d, v, r\in S,$
\begin{align*}
(d\ast v)\dashv \xi(r)&=(d\vdash v+d\bot v)\dashv \xi(r)\\
&=(d\vdash v)\dashv \xi(r)+(d\bot v)\dashv \xi(r)\\
&=\gamma(d)\vdash(v\dashv r)+\gamma(d)\bot(v\dashv r)=\gamma(d)\ast(v\dashv r).
 \end{align*}
This ends the proof.
\end{proof}

\begin{prop}
Given a BiHom-associative supertrialgebra $(S,\dashv,\vdash, \bot, \gamma, \xi)$ and  binary operations are defined by
\begin{align*}
d\ast v&=d\dashv v-(-1)^{|d||v|}v\vdash d\\
\left[d, v\right]&=d\bot v-(-1)^{|d||v|}v\bot d.
 \end{align*}
Then, $(S, \ast, \left[, \right], \gamma, \xi)$ is a BiHom-associative supertrialgebra.
\end{prop}

\begin{proof}
By using Definition \ref{def2}, for each $d, v, r\in S$, we have
\begin{align*}
\left[d, v\right]\ast\xi(r)&=\left[d, v\right]\dashv\xi(r)-\gamma(r)\vdash\left[d, v\right]\\
&=(d\vdash v-(-1)^{|d||v|}v\bot d)\dashv \xi(r)-\gamma(r)\vdash(d\bot v-v\bot v).
 \end{align*}
Moreover,
\begin{align*}
\left[d\ast r,\beta(v)\right]+\left[\alpha(d), v\ast r\right]&=\left[p\dashv r-r\vdash, \beta(v)\right]+\left[\alpha(d), v\dashv r-r\vdash v\right]\\
&=(d\vdash r-r\vdash d)\bot\xi(v)-\gamma(v)\bot(d\dashv r-r\vdash d)\\
&+\gamma(d)\bot(v\dashv r-r\vdash v)-(v\dashv r-r\vdash v)\bot\xi(d).\\
 \end{align*}
Since
\begin{align*}
(d\bot v)\dashv\xi(r)&=\gamma(d)\bot(v\vdash r),\\
(v\bot d)\dashv\xi(r)&=\gamma(v)\bot(d\vdash r),\\
\gamma(r)\vdash (d\bot v)&=(r\vdash d)\bot\xi(v),\\
\gamma(r)\vdash (v\bot d)&=(r\vdash v)\bot\xi(d),\\
(d\bot r)\bot\xi(v)&=\gamma(d)\bot(r\vdash v),\\
\gamma(v)\bot(r\vdash d)&=(v\dashv r)\bot\xi(d)\\
 \end{align*}
in a  BiHom-associative supertrialgebra, we have
$\left[d, v\right]\ast\gamma\xi(r)=\left[d\ast r, \xi(v)\right]+\left[\gamma(d), v\ast r\right].$
Hence,  the proof is completed.
\end{proof}
\begin{prop}
Given a BiHom-associative supertrialgebra $(S, \dashv,\vdash, \bot, \gamma, \xi)$ and  $d\ast v=d\vdash v+d\dashv v+d\bot v$ for any $d, v\in S$.  Then, $(S,\ast,\gamma, \xi)$ is a BiHom-associative superalgebra.

\end{prop}
\begin{proof}
For any $d, v, r\in S,$
\begin{align*}
(d\ast v)\ast \xi(r)&-\gamma(d)\ast(v\ast r)\\
&=(d\vdash v+d\dashv v+d\bot v)\ast\xi(r)-\gamma(d)\ast(v\vdash r+v\dashv r+v\bot w)\\
&=(d\vdash v)\ast\xi(r)+(d\dashv v)\ast\xi(r)+(d\bot v)\ast\xi(r)-\gamma(d)\ast(v\vdash r)-\gamma(d)\ast(v\dashv r)-\gamma(d)\ast(v\bot r)\\
&=(d\vdash v)\vdash\xi(r)+(d\vdash v)\dashv\xi(r)+(d\vdash v)\bot\xi(r)+(d\dashv v)\vdash\xi(r)+(d\dashv v)\dashv\xi(r)\\
&+(d\dashv v)\bot\xi(r)+(d\bot v)\vdash\xi(r)+(d\bot v)\dashv\xi(r)+(d\bot v)\bot\xi(r)-\gamma(d)\vdash(v\vdash r)\\
&-\gamma(d)\dashv(v\vdash r)-\gamma(d)\bot(v\vdash r)
-\gamma(d)\vdash(v\dashv r)-\gamma(d)\dashv(v\dashv r)-\gamma(d)\bot(v\dashv r)\\
&-\gamma(d)\vdash(v\bot r)-\gamma(d)\dashv(v\bot r)-\gamma(d)\bot(v\bot r).
 \end{align*}
This shows that $(S,\ast,\gamma, \xi)$ is a BiHom-associative superalgebra.
\end{proof}

\begin{defn}
 Let $(S, \dashv, \bot, \ast,\gamma, \xi)$ be a BiHom-associative supertrialgebra.    A linear mapping $\lambda : S\longrightarrow S$ with
$\gamma\circ\lambda =\lambda \circ\gamma$ and $\xi\circ\lambda =\lambda \circ\xi$ and for all $d, r\in S, $
\begin{align*}
\lambda (\lambda (d)\bot r)&=\lambda (d)\bot\lambda (r)=\lambda (d\bot \lambda (r)),\\
\lambda (\lambda (d)\dashv r)&=\lambda (d)\dashv\lambda (r)=\lambda (d\dashv \lambda (r)),\\
\lambda (\lambda (d)\vdash r)&=\lambda (d)\vdash\lambda (r)=\lambda (d\vdash \lambda (r))
 \end{align*}
 is called an averaging operator over a BiHom-associative supertrialgebra $S$.
\end{defn}



\section{ Derivations, $(\gamma^{s}\xi^{r})$-derivations and generalized derivations}
\begin{defn}
Let $(S,\dashv,\vdash,\bot,\gamma,\xi)$ be a BiHom-supertrialgebra. A linear mapping $\delta: S \rightarrow S$
is called a derivation if for every $d,v \in S$
\begin{align*}
\delta(d\dashv v)&=\delta(d)\dashv v+d\dashv \delta(v)\\
\delta(d\vdash v)&=\delta(d)\vdash v+ d \vdash \delta(v)\\
\delta(d\bot v)&=\delta(d)\bot v+ d\bot \delta(v)
\end{align*}
and if it holds
\begin{align*}
\gamma\circ \delta=\delta\circ\gamma&,\text{ } \xi\circ \delta=\delta\circ \xi \\
\delta(d\dashv v)&=\delta(d)\dashv\gamma^{s}\xi^{r}(v)+\gamma^{s}\xi^{r}(d)\dashv \delta(v)\\
\delta(d\vdash v)&=\delta(d)\vdash\gamma^{s}\xi^{r}(v)+\gamma^{s}\xi^{r}(d)\vdash \delta(v)\\
\delta(d\bot v)&=\delta(d)\bot\gamma^{s}\xi^{r}(v)+\gamma^{s}\xi^{r}(p)\bot \delta(v),
\end{align*}
it is called an  $(\gamma^{s}\xi^{r})$-derivation of  $(S,\dashv,\vdash,\bot,\gamma,\xi)$.
\end{defn}
The set of $(\gamma^{s}\xi^{r})$-derivation of $S$ is represented by $D_{({\gamma^{s}\xi^{r}})}(S)$ be and define as $$D(S) = (\bigoplus_{s,r\geq 0}D_{({\gamma^{s}\xi^{r}})}(S)_{\overline{0}}+(\bigoplus_{s,r\geq 0}D_{({\gamma^{s}\xi^{r}})}(S)_{\overline{1}}.$$
We show that $D(S)$ has a Lie supertrialgebra structure. For $\delta \in D_{({\gamma^{s}\xi^{r}})}(S)$\; and \;$\delta^{\prime}\in D_{({\gamma^{s^{\prime}}\xi^{r^{\prime}}})}(S)$, we have $[\delta,\delta^{\prime}] \in
D_{({\gamma^{s+s^{\prime}}\xi^{r+r^{\prime}}})}(S),$ where $[\delta,\delta^{\prime}]$ is the standard supercommutator given by $[\delta,\delta^{\prime}] = \delta\bot \delta^{\prime} - (-1)^{|\delta||\delta^{\prime}|}\delta\bot \delta^{\prime}$.

\begin{defn}
 Given   a BiHom-supertrialgebra $(S,\dashv,\vdash,\bot,\gamma,\xi)$ and let $\delta$ be  an endomorphism of $S$. If there exists $\delta^{i}$ for $ i \in \{1,2,3\}$, a family of endomorphisms of $S$ with
 \begin{align*}
\gamma\circ \delta&=\delta\circ\gamma,\text{ } \xi\circ \delta=\delta\circ \xi \\
\gamma\circ \delta^{i}&=\delta^{i}\circ\gamma,\text{ } \xi\circ \delta^{i}=\delta^{i}\circ \xi\\
\delta^{(i)}(u\dashv v)&=\delta^{(i)}(u)\dashv\gamma^{s}\xi^{r}(v)+\gamma^{s}\xi^{r}(u)\dashv \delta^{(i)}(v)\\
\delta^{(i)}(u\vdash v)&=\delta^{(i)}(u)\vdash\gamma^{s}\xi^{r}(v)+\gamma^{s}\xi^{r}(u)\vdash \delta^{(i)}(v)\\
\delta^{(i)}(u\bot v)&=\delta^{(i)}(u)\bot\gamma^{s}\xi^{r}(v)+\gamma^{s}\xi^{r}(u)\bot \delta^{(i)}(v)
\end{align*}
for any $u,v \in S$ and $ i \in \{1,2,3\},$  the linear mapping $\delta$ is called a generalized $(\gamma^{s}\xi^{r})$-derivation  of $S$.
\end{defn}
The set of generalized $(\gamma^{s}\xi^{r})$-derivation of $S$ is represented by $GD(S)$ and as for $D(S)$, we define by
$$GD(S) = (\bigoplus_{s,r\geq 0}GD_{({\gamma^{s}\xi^{r}})})(S)_{\overline{0}}+(\bigoplus_{s,r\geq 0}GD_{({\gamma^{s}\xi^{r}})})(S)_{\overline{1}}.$$

 \begin{defn}
  Given  a BiHom-supertrialgebra $S$.  If there is an endomorphism $\delta^{\prime}$ of $S$ with
\begin{align*}
\gamma\circ \delta&=\delta\circ\gamma,\text{ } \xi\circ \delta=\delta\circ \xi \\
\gamma\circ \delta^{\prime}&=\delta^{\prime}\circ\gamma,\text{ } \xi\circ \delta^{\prime}=\delta^{\prime}\circ \xi \\
\delta^{\prime}(d\dashv v)&=\delta(d)\dashv\gamma^{s}\xi^{r}(v)+\gamma^{s}\xi^{r}(d)\dashv \delta(v)\\
\delta^{\prime}(d\vdash v)&=\delta(d)\vdash\gamma^{s}\xi^{r}(v)+\gamma^{s}\xi^{r}(d)\vdash \delta(v)\\
\delta^{\prime}(d\bot v)&=\delta(d)\bot\gamma^{s}\xi^{r}(v)+\gamma^{s}\xi^{r}(d)\bot \delta(v)
\end{align*}
for any $d,v \in S$, then  an endomorphism $\delta$ of $S$ is said to be a $\gamma^{s}\xi^{r}$-quasiderivation. \\
We then define
$$QD(S) = (\bigoplus_{s,r\geq 0}QD_{({\gamma^{s}\xi^{r}})})(S)_{\overline{0}}+(\bigoplus_{s,r\geq 0}QD_{({\gamma^{s}\xi^{r}})})(S)_{\overline{1}}.$$	
\end{defn}
\begin{rem}
We have
$$D(S) \subseteq QD(S) \subseteq GD(S).$$
\end{rem}

\begin{defn}
Given   a BiHom-supertrialgebra $(S,\dashv,\vdash,\bot,\gamma,\xi)$. If
for $d,v\in S$, we have \begin{equation}
     \begin{aligned}
         \delta(d\dashv v)=& \delta(d)\dashv v= 0,\\\delta(d\vdash v)= &\delta(d)\vdash v= 0,\\ \delta(d\bot v)=& \delta g(d)\bot v= 0,  \end{aligned}
 \end{equation}
 then we say a linear mapping $\delta$ is  an $(\gamma^{s}\xi^{r})$-central derivation of $S$.
\end{defn}
The set of all central derivations and $(\gamma^{s}\xi^{r})$-central derivations of $S$ are represented by $ZD(S)$ and $ZD_{({\gamma^{s}\xi^{r}})}(S)$ respectively.
$$ZD(S) = (\bigoplus_{s,r\geq 0}ZD_{({\gamma^{s}\xi^{r}})})(S)_{\overline{0}}+(\bigoplus_{s,r\geq 0}ZD_{({\gamma^{s},\xi^{r}})})(S)_{\overline{1}}.$$
\begin{prop}$ZD(S)$ is an ideal of $D(S)$.\end{prop}\begin{proof} Firstly, we need to show  that $ZD(S)$ is a subsuperalgebra of $D(S)$. For $ \delta_1\in ZD(S)$ and $ \delta_2\in D(S) $, we have
\begin{align*}\delta_1\delta_2(d\circ v)= \delta_1(\delta_2(d)\circ v + d\circ \delta_2(v))= \delta_1(\delta_2(d)\circ v )+\delta _1(d\circ \delta_2(v))=0.
\end{align*}
Similarly, for $\delta _2\delta_1(d\circ v)=0. $ It completes the proof.
\end{proof}


\begin{defn}
  The set of linear mappings $\delta$ such that for every $p, v\in S$,
\begin{align*}
\gamma\circ \delta&=\delta \circ\gamma,\quad \xi\circ \delta=\delta\circ\xi,\\
\delta(d)\dashv \gamma\xi(v)&=\delta(d)\dashv \delta(v)=\gamma\xi(d)\dashv \delta(v),\\
 \delta(d)\vdash \gamma\xi(v)&=\delta(d)\vdash \delta(v)=\gamma\xi(d)\vdash \delta(v),\\
 \delta(d)\bot \gamma\xi(v)&=\delta(d)\bot \delta(v)=\gamma\xi(d)\bot \delta(v)
\end{align*}
is called $(\gamma^{s}\xi^{r})$-centroid of  $(S,\dashv,\vdash,\bot,\gamma,\xi)$ and it is denoted by $C_{(\gamma^{s}\xi^{r})}(S)$.
We set
$$C(S) = (\bigoplus_{s,r\geq 0}C_{({\gamma^{s}\xi^{r}})})(S)_{\overline{0}}+(\bigoplus_{s,r\geq 0}C_{({\gamma^{s}\xi^{r}})})(S)_{\overline{1}}.$$
\end{defn}
\begin{defn}
Let $A$ be a nonempty subset of a BiHom-associative supertrialgebra $S$. This set
\begin{align*}
Z_{S}(A)=\left\{u\in A\, |\, \gamma\xi(u)\ast A = A\ast\gamma\xi(u)=0\right\}
\end{align*}
is said to be the centralizer of $A$ in $S$ with the $\ast=\dashv, \vdash$ and $\bot$ respectively.\\
 \end{defn}
 Clearly, $Z_{S}(S)=Z(S)$ is the center of $S.$
\begin{prop}
For any r,s, we have
$$ZD_{({\gamma^{s}\xi^{r}})}(S) =D_{({\gamma^{s}\xi^{r}})}(S) \cap C_{({\gamma^{s}\xi^{r}})}(S).$$
\end{prop}
\begin{proof}
Let $\delta$ be an of $ZD_{({\gamma^{s}\xi^{r}})}(S)$. Then $g$ is a central derivation, this means that it holds the following conditions
\begin{align*}
\delta(d\vdash v)&=\delta(d)\vdash v=0\\
\delta(d\dashv v)&=\delta(d)\dashv v=0\\
\delta(d\bot v)&=\delta(d)\bot v=0
\end{align*}
for $d,v \in S$. Since $g$ is a central derivation, it is also a $({\gamma^{s}\xi^{r}})$-derivation. Thus, $\delta \in D_{({\gamma^{s}\xi^{r}})}(S)$.
Moreover, as $\delta$ is central, it is also a $({\gamma^{s}\xi^{r}})$-central mapping. Therefore, $\delta \in C_{({\gamma^{s}\xi^{r}})}(S)$.
Hence, $ZD_{({\gamma^{s}\xi^{r}})}(S) =D_{({\gamma^{s}\xi^{r}})}(S) \cap C_{({\gamma^{s}\xi^{r}})}(S).$
Now, take $\delta\in D_{({\gamma^{s}\xi^{r}})}(S) \cap C_{({\gamma^{s}\xi^{r}})}(S).$ Hence, $\delta$ is both a $({\gamma^{s}\xi^{r}})$-derivation and a $({\gamma^{s}\xi^{r}})$-central mapping. Since $\delta$ is a $({\gamma^{s}\xi^{r}})$-derivation, it satisfies Leibniz's rule and the below conditions
\begin{align*}
\delta(d\vdash v)&=\gamma^{s}(\delta(d))\vdash v+d \vdash \delta(v)=0\\
\delta(d\dashv v)&=\gamma^{s}(\delta(d))\dashv v+d \dashv \delta(v)=0\\
\delta(d\bot v)&=\gamma^{s}(\delta(d))\bot v+d \bot \delta(v)=0
\end{align*}
for $d,v \in S$. Also, as $\delta$ is a $({\gamma^{s}\xi^{r}})$-central mapping, it commutes with $\gamma^{s}$ and $\xi^{r}$.
Thus, $\delta$ satisfies all the conditions to be a $({\gamma^{s}\xi^{r}})$-central derivation, this means that $g \in ZD_{({\gamma^{s}\xi^{r}})}(S)$.
Therefore, $D_{({\gamma^{s}\xi^{r}})}(S) \cap C_{({\gamma^{s}\xi^{r}})}(S) \subseteq ZDer_{({\gamma^{s}\xi^{r}})}(S)$. Combining both inclusions, we have  $$ZD_{({\gamma^{s}\xi^{r}})}(S) =D_{({\gamma^{s}\xi^{r}})}(S) \cap C_{({\gamma^{s}\xi^{r}})}(S).$$
\end{proof}

\begin{lem}
Given a BiHom-supertrialgebra $(S,\dashv,\vdash,\bot,\gamma,\xi)$. Then
\item [(i)]$[D(S),C(S)] \subseteq C(S)$,
\item [(ii)]$ C(S)\oplus D(S) \subseteq D(S)$.
\end{lem}
\begin{proof}
 \text{(i)}    To prove this, we need to show that the commutator of any element in $D(S)$
 with any element in $C(S)$ is also in $C(S)$.
 We take $\delta\in D(S)$ and $c\in C(S)$. Then, we have
 $\delta(d)\vdash c(v)-c(d)\vdash \delta(v) \in C(S)$.
 Since $\delta \in D(S)$ and $c \in C(S),$ by using definitions we rewrite as
 \begin{align*}
\delta(d)\vdash c(v)-c(d)\vdash \delta(v)  &=(\delta(d)\vdash \gamma \xi(c(v))-\gamma \xi(c(d))\vdash \delta(v))\\
 &=\gamma \xi(\delta(d))\vdash c(v))-(c(d))\vdash \gamma \xi (\delta(v)))\\
 &=\gamma \xi(\delta(d))\vdash \delta(c(v)))-(\delta(c(d))\vdash \gamma \xi(\delta(v)))\\
 &= \delta(d)\vdash \gamma \xi \delta(v)))-(\gamma \xi \delta(d))\vdash \delta(v))\\
 &= \delta(d)\vdash c(v)-c(d)\vdash \delta(v).
 \end{align*}
 Hence, the commutator of $\delta$ and $c$ is also in $C(S)$, as required.\\
\text{(ii)} Let $c \in C(S)$ and $f \in D(S)$. Then, we have
$\delta(d\vdash c(v))=\delta(d)\vdash c(v)+d\vdash \delta(c(v)).$
By using the fact that $c \in C(S)$,  the right side is expressed by
\begin{align*}
 &=(\delta(d)\vdash \gamma \xi(c(v))+\gamma \xi(d)\vdash \delta(c(v))\\
 &= (\delta(d))\vdash \delta(c(v)))+(\gamma \xi(d)\vdash \delta(c(v)))\\
 &= (\delta(d\vdash \delta(v)))+(\gamma \xi(d)\vdash \delta(c(v)))\\
 &= (\delta(d\vdash \delta(v)))+(\delta(\gamma \xi(d)\vdash c(v)))\\
 &= \delta(d)\vdash \delta(v)+\gamma \xi(d)\vdash c(v)).
\end{align*}
Similary, we can prove the same for the other cases.
Therefore, $ C(S)\oplus D(S) \subseteq D(S)$, which proves (ii).
\end{proof}

\begin{defn}
  The set of linear mappings $\delta$ with
$$ d\circ \delta(v)= \delta(d)\circ v$$ for all $d,v \in S$ where $\circ =\dashv,\vdash$ and $\bot$ respectively  is said to be the $(\gamma^{s}\xi^{r})$-quasicentroid $QC_{(\gamma^{s}\xi^{r})}(S)$.
We set
$$QC(S) = (\bigoplus_{s,r\geq 0}QC_{({\gamma^{s}\xi^{r}})})(S)_{\overline{0}}+(\bigoplus_{s,r\geq 0}QC_{({\gamma^{s}\xi^{r}})})(S)_{\overline{1}}.$$
\end{defn}

 \begin{lem}
 Given a BiHom-supertrialgebra $(S,\dashv,\vdash,\bot,\gamma,\xi)$.
\item [(i)]$[QD(S),QC(S)] \subseteq QC(S)$,
\item [(ii)]$ C(S) \subseteq QD(S)$,
\item [(iii)]$[QC(S),QC(S)] \subseteq QD(S)$,
\item [(iv)]$QD(S) + QC(S) \subseteq GD(S)$
\end{lem}

\begin{proof}
\text{(i)} Let $\delta \in  QD(S)$ and $c \in QC(S)$. Then, by using definition, we have
 \begin{align*}
\delta(d\dashv v)&=\delta(p)\dashv \gamma^{s}\xi^{r}(v)+\gamma^{s}\xi^{r}(d) \dashv \delta(v)\\
c(d\dashv v)&=\gamma\xi(d)\dashv \delta(v)-\delta(d)\dashv\gamma\xi(v).
\end{align*}
Now, we compute the commutator $[\delta,c]$ as
\begin{align*}
[\delta,c](d)&=\delta(c(d))-c(\delta(d))\\
&=\delta(\gamma\xi(d)\dashv \delta(v)-\delta(d)\dashv\gamma\xi(v))-(\gamma\xi(d)\dashv \delta(v)-\delta(d)\dashv\gamma\xi(v))\dashv\gamma^{s}\xi^{r}(v)\\
&-\gamma^{s}\xi^{r}(d)\dashv(\gamma\xi(v)\dashv \delta(d)-\delta(v)\dashv\gamma\xi(d)).
\end{align*}
Using the properties of quasiderivations and centroids, we can simplify this expression
\begin{align*}
&\delta(\gamma\xi(d))\dashv \delta(v)-\delta(\delta(d))\dashv\gamma\xi(v)-(\gamma\xi(d)\dashv \delta(\delta(v))-\delta(d)\dashv \delta(\gamma\xi(v)))\dashv\gamma^{s}\xi^{r}(v)-\gamma^{s}\xi^{r}(p)\\
&\dashv(\gamma\xi(v)\dashv
\delta(d)-\delta(v)\dashv\gamma\xi(d))\\
&=\delta(\gamma\xi(d))\dashv \delta(v)-\delta(\delta(d))\dashv\gamma\xi(v)- \gamma\xi(d)\dashv \delta(\delta(v))+\delta(d)\dashv \delta(\gamma\xi(v))-\gamma^{s}\xi^{r}(v)\dashv\gamma\xi(d)\\
&\dashv \delta(d)+\gamma^{s}\xi^{r}(d)\dashv \delta(v)-\gamma\xi(v) \dashv\gamma^{s}\xi^{r}(d)\dashv \delta(d)+\gamma\xi(d)\dashv\gamma^{s}\xi^{r}(v)\dashv \delta(v)\\
&=\delta(\gamma\xi(d))-\gamma^{s}\xi^{r}(d)\dashv \gamma \xi(v))\dashv \delta(v)-(\delta(\delta(d))-\gamma^{s}\xi^{r}(v)\dashv\gamma\xi(d))\dashv\gamma\xi(v)-(\gamma\xi(d)\dashv \delta(\delta(v))-\delta(d)\\
&\dashv \delta(\gamma\xi(v)))\dashv \gamma^{s}\xi^{r}(v)-(\gamma\xi(v)\dashv \gamma^{s}\xi^{r}(d)\dashv \delta(d)-
\gamma\xi(d)\dashv \gamma^{s}\xi^{r}(v) \dashv \delta(v)).
\end{align*}
 This expression shows that $[\delta,c] \in QC(S)$.
  Hence, $[QD(S),QC(S)] \subseteq QC(S).$\\
\text{(ii)} Here our aim is to prove that centroids are contained within quasiderivations, i.e., $ C(S) \subseteq QD(S)$.
Let $c \in C(S)$. By using definition, $c$ satisfies certain properties, which we can leverage to show that it indeed belongs to the set of quasiderivations.
Given the properties of centroids, we can choose appropriate $\delta^{\prime}$ such that
\begin{align*}
\delta^{\prime}(d\dashv v)&=c(d)\dashv\gamma^{s}\xi^{r}(v)+\gamma^{s}\xi^{r}(d)\dashv c(v),\\
\delta^{\prime}(d\vdash v)
&=c(d)\vdash\gamma^{s}\xi^{r}(v)+\gamma^{s}\xi^{r}(d)\vdash c(v)
\end{align*}
and $$\delta^{\prime}(d\bot v)=c(d)\bot\gamma^{s}\xi^{r}(v)+\gamma^{s}\xi^{r}(d)\bot c(v).$$
Therefore, $c$ can be represented as a quasiderivation $\delta^{\prime}$. Hence $ C(S) \subseteq QD(S)$.\\
\text{(iii)} To prove that the commutator of centroids lies within quasiderivations, i.e.,
$[QC(S),QC(S)] \subseteq QD(S)$, we can follow a similar approach to (i).\\
\text{(iv)} To show that $QD(S) + QC(S) \subseteq GD(S)$, we can use that $QD(S)$ and $C(S)$ connected with $GD(S)$.
Therefore, the proof is completed.
\end{proof}
\begin{prop}
 If $(S,\dashv,\vdash,\bot,\gamma,\xi)$ is a BiHom-supertrialgebra with trivial center,
then we have
$$D(S)\oplus C(S) \subseteq QD(S).$$
\end{prop}
\begin{proof}
 $D(S)$ and $C(S)$ are subspaces of $QD(S)$.
Furthermore, if $\delta \in D(S)\cap
C(S)$, then for $d, v \in S$, $\delta([d,v]) = 0$ for
every $d,v \in S$. Therefore $\delta(d) \in Z(S)$, hence $\delta= 0$.
\end{proof}



	
\end{document}